\documentclass[10pt, a4paper]{amsart}
\usepackage{amsmath,amssymb,hyperref}

\usepackage{tikz}
\usetikzlibrary{cd,arrows.meta,calc,decorations.markings,patterns}

\usepackage[bb=ams,scr=euler]{mathalpha}
\usepackage{enumitem}
\usepackage{thmtools}
\declaretheoremstyle[headfont=\normalsize\normalfont\bfseries,notefont=\mdseries,
notebraces={(}{)},bodyfont=\normalfont\itshape,postheadspace=0.5em]{italstyle}

\declaretheorem[style=italstyle,name=Theorem]{theorem}
\declaretheorem[style=italstyle,name=Lemma,sibling=theorem]{lemma}
\declaretheorem[style=italstyle,name=Corollary,sibling=theorem]{corollary}
\declaretheorem[style=italstyle,name=Question]{question}
\declaretheorem[style=italstyle,name=Conjecture,sibling=question]{conjecture}

\makeatletter
\newcommand{\abs}[1]{\left|#1\right|}
\newcommand{\bd}{\partial}
\newcommand{\C}{\mathbb{C}}
\renewcommand{\d}{\mathrm{d}}

\newcommand{\R}{\mathbb{R}}
\newcommand{\set}[1]{\left\{#1\right\}}
\newcommand{\Z}{\mathbb{Z}}

\renewcommand\section{\@startsection{section}{1}{0pt}{-3.5ex \@plus -1ex \@minus -.2ex}{2.3ex \@plus.2ex}{\centering\bfseries}}
\renewcommand{\subsection}{\@startsection{subsection}{2}\z@{.5\linespacing\@plus.7\linespacing}{-.5em}{\normalfont\bfseries}}
\makeatother

\author{Dylan Cant}
\email{dylan@dylancant.ca}
\address{Institut de mathématique d'Orsay, Université Paris-Saclay, Bâtiment 307, rue Michel Magat, F-91405 Orsay Cedex, France}
\date{\today}
\begin{document}
\title{Hamiltonian linking and Symplectic packing}
\begin{abstract}
  In this short note, we make an observation relating symplectic packings of the standard symplectic ball by two sets and Hamiltonian linking.
\end{abstract}
\maketitle

\section{Introduction}
\label{sec:introduction}

\subsection{Statement of results}
\label{sec:statement-results}

Denote by $G$ be the group of compactly supported Hamiltonian diffeomorphisms of $\R^{2n}$, i.e., the group of time-1 maps of compactly supported Hamiltonian isotopies.

Let us say that two disjoint compact subsets $K_{1},K_{2}\subset \R^{2n}$ are \emph{Hamiltonian unlinked} provided there is $\varphi\in G$ such that:
\begin{equation*}
  \varphi(K_{1})\text{ and }\varphi(K_{2})\text{ lie on opposite sides of a hyperplane}.
\end{equation*}
Otherwise we say they are \emph{Hamiltonian linked}.

There is an obvious obstruction to being Hamiltonian unlinked, namely the two sets must be unlinkable via a compactly supported smooth isotopy.

Denote by $c(K_{i})$ the \emph{spectral capacity of $K_{i}$}, which is simply the Floer theory version\footnote{Generating function capacities are known to equal Floer homology capacities; see \cite{viterbo-arXiv-1996,hermann-BSMF-2004}.} of Viterbo's generating function capacity introduced in \cite{viterbo-mathann-1992}; in this paper we follow the conventions of \cite{alizadeh-atallah-cant-arXiv-2024,cant-zhang-arXiv-2024}. This is a normalized symplectic capacity. Then:
\begin{theorem}\label{theorem:main}
  Suppose $K_{1},K_{2}\subset B(a)$ are compact disjoint subsets of the standard symplectic ball (of capacity $a>0$). Either:
  \begin{enumerate}
  \item $K_{1},K_{2}$ are Hamiltonian linked,
  \item $c(K_{1})+c(K_{2})\le a$;
  \end{enumerate}
  in other words, if two sets are Hamiltonian unlinked, then they obey a packing inequality; if two sets break the packing inequality, then they must be Hamiltonian linked.
\end{theorem}
\emph{Remark}. The same result holds if $B(a)$ is replaced by a domain $\Omega$ satisfying $a=c(\Omega)=\gamma(\Omega)$, where $\gamma(\Omega)$ is the \emph{spectral diameter} (i.e., the diameter of the space of isotopies supported in $\Omega$ with respect to Viterbo's spectral metric). In \cite{alizadeh-atallah-cant-arXiv-2024} it is shown that $a=c(B(a))=\gamma(B(a))$.

\begin{corollary}\label{corollary:linking-small-ball}
  Let $K\subset \bd B(a)$ be a compact set with capacity $c(K)=b$ (there exist examples where $K$ is a Lagrangian and $b=a/2$). Then the smaller ball $B(c)\subset B(a)$ is Hamiltonian linked with $K$ if $c>a-b$.\hfill$\square$
\end{corollary}
In dimension $n=1$, this follows from smooth topology. In higher dimensions $n>1$, the sets $K$ and $B(c)$ can be chosen to be smoothly unlinked, and so it produces a camel type theorem: {\itshape the ball $B(c)$ cannot be isotoped to $v+B(c)$ through balls of capacity $c$ in the complement of $K$, if $v$ is sufficiently large.}

Let us say that a compact domain $\Omega\subset \C^{n}$ is \emph{boundary minimal} if:
\begin{equation*}
  c(\Omega\setminus U)< c(\Omega)
\end{equation*}
for all open sets $U$ such that $U\cap \bd \Omega\ne \emptyset$. Then we recover a result of \cite{ziltener-imrn-2016} (in the case $\Omega=B(a)$); his argument used displacement energy while ours uses the packing versus linking theorem.
\begin{corollary}\label{cor:gamma-is-c-implies-bm}
  Any domain $\Omega$ satisfying $c(\Omega)=\gamma(\Omega)$ is boundary minimal.
\end{corollary}
\begin{proof}[Proof of Corollary \ref{cor:gamma-is-c-implies-bm}]
  If $U$ is open and has non-empty intersection with $\bd \Omega$, then inside of $U\cap \Omega$ one can find a small ball which is unlinked with $\Omega\setminus U$. The result then follows from the packing inequality.
\end{proof}
Boundary minimality is interesting as it implies the following closing-type lemma (see \cite{irie-JMD-2015,chaidez-datta-prasad-tanny-JMD-2024,xue-arxiv-2022} for similar results):
\begin{lemma}\label{lemma:closing}
  Suppose that $\Omega\subset \C^{n}$ is a symplectically embedded Liouville domain which is boundary minimal. Let $\alpha$ be the intrinsic contact form on $\bd\Omega$ (using the Liouville domain structure of $\Omega$). Then, for any non-empty open subset $V\subset \bd\Omega$ and any non-negative non-constant function $f$ supported in $V$, there is a Reeb orbit $\gamma$ for the Reeb flow associated to $e^{-sf}\alpha$ which passes through $V$, for some $s\in (0,1]$. Moreover:
  \begin{equation*}
    \text{period of }\gamma+\Gamma(\gamma)<c(\Omega)
  \end{equation*}
  where $\Gamma$ is the cohomology class $[\lambda_{\C^{n}}-\lambda_{\Omega}]\in H^{1}_{\mathrm{dR}}(\Omega)$. In particular, if the embedding is exact, one can bound the period of $\gamma$ from above and conclude there is an orbit of $\alpha$ with period $\le c(\Omega)$ passing through any point of $\bd\Omega$.
\end{lemma}
This is proved in \S\ref{sec:proof-lemma-closing}. As explained to the author by S.\@~Matijevi\'c, the last sentence of the lemma, together with the results of \cite{abbondandolo-kang-duke-2022,abbondandolo-benedetti-GAFA-2023}, implies the following characterization of uniformly convex Zoll domains:
\begin{theorem}\label{theorem:convex-zoll}
  A uniformly convex domain $\Omega$ is Zoll if and only if it is boundary minimal.
\end{theorem}
\begin{proof}
  If a uniformly convex domain is boundary minimal, Lemma \ref{lemma:closing}, and the fact that $c(\Omega)$ is the shortest period of a Reeb orbit (proved in \cite{abbondandolo-kang-duke-2022}) imply that $\Omega$ is Zoll. Conversely, \cite{abbondandolo-benedetti-GAFA-2023} proved that a uniformly convex Zoll domain is a local maximizer for the systolic ratio. Shrinking the domain slightly will decrease the volume, and so the systole must also decrease. Since small variations remain uniformly convex, it follows (again from \cite{abbondandolo-kang-duke-2022}) that $c(\Omega)$ decreases, as desired.
\end{proof}

Our next result concerns the areas of $J$-holomorphic disks with boundary on Lagrangians. Here $J$ is an $\omega$-tame (and standard at infinity) almost complex structure (we call such a $J$ \emph{admissible}).
\begin{corollary}\label{corollary:lagrangian-hbar}
  Let $K\subset \bd B(a)$ be a subset with capacity $b$. Then each Lagrangian $L$ in the interior of $B(a)$ either:
  \begin{enumerate}
  \item bounds a $J$-holomorphic disk with area at most $a-b$,
  \item is linked with $K$,
  \end{enumerate}
  where $J$ is any admissible almost complex structure.
\end{corollary}
When $n=1$ the result is trivial as every subset of $B(a)$ is smoothly linked with $K$, or $b=0$. On the other hand, when $n>1$, the result is symplectic in nature, as $K$ and $L$ are smoothly unlinked unless $K\ne \bd B(a)$. We prove Corollary \ref{corollary:lagrangian-hbar} in \S\ref{sec:proof-corollary-lag-hbar}.

\emph{Remark}. Perhaps, when $n>1$, every Lagrangian $L$ in the interior of $B(a)$ is unlinked with some $K$ with $c(K)\ge a/2$, in which case one proves the conjecture about the areas of disks being less than $a/2$; see \S\ref{sec:further-questions}.

Our next result discusses a case when sets are automatically unlinked.
\begin{lemma}\label{lemma:starshaped-unlinked}
  If $K_{1}$ and $K_{2}$ are disjoint compact sets, each symplectomorphic to a starshaped set, then they are Hamiltonian unlinked.
\end{lemma}
This is shown in \S\ref{sec:proof-lemma-starshaped-unlinked}. Theorem \ref{theorem:main} then yields:
\begin{corollary}
  Suppose $a=c(\Omega)=\gamma(\Omega)$. Two disjoint sets $K_{1},K_{2}\subset \Omega$, each of which is symplectomorphic to a starshaped set, satisfy the packing inequality $c(K_{1})+c(K_{2})\le a$.\hfill$\square$
\end{corollary}
In fact, the proof of Theorem \ref{theorem:main} shows that two disjoint sets $K_{1},K_{2}$ which obey the maximum formula for spectral invariants also obey the packing inequality. It is known that the class of pairs of subsets of $\R^{2n}$ which obey the maximum formula contains all pairs of the images of two embeddings of Liouville domains which are:
\begin{enumerate}
\item exact \cite{humiliere-le-roux-seyfaddini-2016}, or
\item $\pi_{1}$-injective on the boundary \cite{ganor-tanny-AGT-2023},
\end{enumerate}
We note that \cite{humiliere-le-roux-seyfaddini-2016} also proves the the maximum formula for the generating function capacity for a pair of subsets of $\R^{2n}$ which are separated by a hyperplane.


\subsection{Discussion of proof}
\label{sec:discussion-proof}

The key ingredients used to deduce Theorem \ref{theorem:main} are:
\begin{enumerate}
\item The fact that the \emph{spectral diameter of a ball equals its capacity}, i.e., $\gamma(B(a))=c(B(a))=a$, proved in \cite{alizadeh-atallah-cant-arXiv-2024}.
\item The maximum formula for spectral invariants, proved in \cite{humiliere-le-roux-seyfaddini-2016,tanny-JSG-2022,ganor-tanny-AGT-2023}. 
\end{enumerate}
To deduce Corollary \ref{corollary:lagrangian-hbar}, one also appeals to the result in \cite{hermann-BSMF-2004,cant-zhang-arXiv-2024}.

\subsection{Further questions}
\label{sec:further-questions}

\subsubsection{Areas of holomorphic disks}
\label{sec:areas-holom-disks}

Corollary \ref{corollary:lagrangian-hbar} suggests the question:
\begin{question}
  Is every Lagrangian $L$ in the interior of $B(a)$ Hamiltonian unlinked with some compact subset $K\subset \bd B(a)$ satisfying $c(K)\ge a/2$?
\end{question}
As explained above, if the answer is ``yes,'' then it would solve the following conjecture (compare also with \cite{cieliebak-mohnke-inventiones-2018}):
\begin{conjecture}\label{conjecture:lagrangian-hbar}
  Every Lagrangian $L$ in $B^{2n}(a)$ bounds $J$-holomorphic disks of area at most $a/2$, for all admissible $J$, if $n>1$.
\end{conjecture}

\subsubsection{Minimal sets}
\label{sec:minimal-sets}

An interesting direction raised by Corollary \ref{corollary:linking-small-ball} is that $\bd B(a)$ is a \emph{minimal compact set} for the spectral capacity: any compact subset has a strictly smaller spectral capacity. Moreover, as explained above, the domain $B(a)$ is a boundary minimal Liouville domain, which implies some version of the closing lemma (see Lemma \ref{lemma:closing}).
\begin{question}
  Which compact sets in $\R^{2n}$ are minimal in the above sense? Which domains $\Omega$ are boundary minimal?
\end{question}
It is not hard to see using \cite{laudenbach-sikorav-IMRN-1994} and \cite{hermann-BSMF-2004,cant-zhang-arXiv-2024} that any compact connected Lagrangian submanifold is minimal, since removing a small ball will change its capacity from a positive number to zero. However, the boundary of a polydisk is never minimal.

Turning now to Zoll domains, Theorem \ref{theorem:convex-zoll} suggests the question:
\begin{question}\label{question:spectral-characterization}
  Does every (uniformly convex) Zoll domain $\Omega$ satisfy $\gamma(\Omega)=c(\Omega)$?
\end{question}
If the answer is ``yes,'' then $\gamma(\Omega)=c(\Omega)$ gives an alternative ``spectral'' characterization of when a (uniformly convex) domain is Zoll. It is known by the work of \cite{ginzburg-gurel-mazzucchelli-AIHPAN-2021,matijevic-arXiv-2024,matijevic-arXiv-2025} that a uniformly convex domain is Zoll if and only if the first $n$ Gutt-Hutchings capacities agree (see \cite{gutt-hutchings-agt-2018} for the definition of these capacities). This raises a variant of Question \ref{question:spectral-characterization}:
\begin{question}
  Does every (uniformly convex) domain $\Omega$ with equal first and $n$th Gutt-Hutchings capacities satisfy the equality $\gamma(\Omega)=c(\Omega)$?
\end{question}

The notion of minimality also raises the question of how much the capacity drops when an open set is removed? On this topic, we prove the following:
\begin{theorem}\label{theorem:toric}
  If $0<k<n$, and $0\le b_{1}\le b_{2}\le a$ then: $$K_{k,b_{1},b_{2}}:=\bd B(a)\cap \set{b_{1}\le \sum_{i=1}^{k} \pi\abs{z_{i}}^{2}\le b_{2}}$$ satisfies $c(K_{k,b_{1},b_{2}})=\min\set{b_{2},a-b_{1}}$.
\end{theorem}
In particular, even though the boundary is minimal, it contains proper compact subsets of any given smaller capacity. The proof is a simple application of Poisson commuting partitions of unity (i.e., using the toric structure) and embedding tricks, and is given in \S\ref{sec:proof-theorem-toric}.

\subsubsection{Strong Arnol'd chord conjecture}
\label{sec:strong-arnold-chord}

Boundary minimal sets are also related to the strong Arnol'd chord conjecture (see, e.g., \cite{ziltener-imrn-2016,kang-arxiv-2023,kang-zhang-arXiv-2024}). Indeed, adapting the argument of \cite{ziltener-imrn-2016}, we can prove:
\begin{theorem}[Ziltener]
  Suppose that that $\Omega\subset \R^{2n}$ is a boundary minimal exactly embedded Liouville domain, and $c(\Omega)$ is the minimal period of the Reeb flow on $\bd\Omega$ (e.g., if $\Omega$ is a uniformly convex Zoll domain). Let $\Lambda\subset \bd \Omega$ be a Legendrian. Then $\Lambda$ has a Reeb chord of length at most $c(\Omega)/2$.
\end{theorem}
\begin{proof}
  Consider the map $\Lambda\times \R/c(\Omega)\Z\to \bd \Omega$ which evolves $\Lambda$ by the Reeb flow. We have shown already in Lemma \ref{lemma:closing} that the Reeb flow on $\bd\Omega$ is $c(\Omega)$-periodic, so the map is well-defined. If this map were injective, then it would produce an embedded Lagrangian $L\subset \bd \Omega$ with $\hbar(L)=c(\Omega)$, as in \cite{mohnke-annals-2001,ziltener-imrn-2016}. Then by \cite{hermann-BSMF-2004,cant-zhang-arXiv-2024} this implies $c(L)=c(\Omega)$. However, $L\subset \bd \Omega$ is a proper compact subset of $\bd \Omega$, and so $c(L)<c(\Omega)$ by boundary minimality, thus the map could not have been an embedding. Thus there are two distinct points $x_{1},x_{2}\in \Lambda$ which lie on the same Reeb orbit, and this implies a Reeb chord of length $\le c(\Omega)/2$.
\end{proof}
Incidentally, using the projection $\pi:\bd B(1)\to \mathbb{CP}^{n-1}$, the fact that Lagrangians in $\bd B(1)$ (for $n>1$) bound $J$-holomorphic disks of area strictly less than $1$ implies any Lagrangian in $K\subset \mathrm{CP}^{n-1}$ bounds a disk of area strictly less than $1$.\footnote{Here we note that holomorphic disks of area less than $1$ with boundary in $\bd B(1)$ miss the origin (by Gromov's monotonicity theorem \cite{gromov-inventiones-1985}) and hence project to smooth disks in $\mathbb{C}P^{n-1}$ under the projection $\C^{n}\setminus \set{0}\to \mathbb{C}P^{n-1}$.} Since $K$ also bounds disks with symplectic area $1$, it follows that $K$ bounds a smooth disk with area less than $1/2$. This disk can be lifted to $\pi^{-1}(K)\subset \bd B(1)$. Thus we conclude:
\begin{corollary}
  Every Lagrangian in $\bd B(1)$, with $n>1$, bounds a smooth disk contained in $\bd B(1)$ with symplectic area at most $1/2$.
\end{corollary}
This result is certainly known to Ziltener (see \cite{ziltener-imrn-2016}), and presumably is known by other researchers as well.
This is related to Conjecture \ref{conjecture:lagrangian-hbar}, but does not imply it, as the disk is not a priori representable by a $J$-holomorphic disk for every admissible $J$; moreover the argument applies only to Lagrangians contained in the boundary of the ball.

The \emph{rationality constant} $\rho(L)$ of a Lagrangian $L$ is infimal positive area of a smooth disk with boundary on $L$. Thus we have shown that $\rho(L)\le a/2$ for all $L\subset \bd B(a)$ when $n>1$. It is natural to compare with \cite{dimitroglou-rizell-proc-GGT-2015,faisal-arXiv-2025} and ask the question:
\begin{question}
  Does every Lagrangian in $L\subset B(a)$ satisfy $\rho(L)\le a/2$? If yes, are the Lagrangians $L$ with $\rho(L)=a/2$ automatically contained in $\bd B(a)$?
\end{question}
The question is only posed in $\R^{2n}$ for $n>1$, and the answer is known to be yes in both cases, in $\R^{4}$, by the work of \cite{cieliebak-mohnke-inventiones-2018,dimitroglou-rizell-proc-GGT-2015,faisal-arXiv-2025}.

Another consequence of the packing versus linking dichotomy is:
\begin{theorem}\label{theorem:chord-1-2}
  Suppose that $\Lambda_{0},\Lambda_{1}\subset \bd B(1)$ are disjoint and unlinked Legendrians, and let $\alpha$ denote the standard contact form on $\bd B(1)$. Then, for any $f\ge 0$, there exists a chord from $\Lambda_{0}\cup \Lambda_{1}$ to $\Lambda_{0}\cup \Lambda_{1}$ of length at most $1/2$ with respect to the contact form $e^{-f}\alpha$.
\end{theorem}
Here two Legendrians in $\bd B(1)$ are unlinked provided there is a Legendrian isotopy of the link $\Lambda_{0}\cup \Lambda_{1}$ sending the components to opposite sides of a hyperplane. We give the proof in \S\ref{sec:proof-theorem-chord-1-2}. This result is interesting because it is stable with respect to a class of $C^{0}$ small perturbations of the ball (perturbations which push the ball inwards). It is known that the conclusion of the strong Arnol'd conjecture is not stable under $C^{0}$ small pertubations; see the discussion in \cite{kang-arxiv-2023,kang-zhang-arXiv-2024}.

\subsubsection{Can zero capacity set link the ball?}
\label{sec:can-zero-capacity}

\begin{question}
  If $K$ links the ball $B(a)$, is it necessary that $c(K)>0$?
\end{question}
If the answer is positive, then one resolves the question of Ginzburg asking whether or not it holds that $c(N)>0$ for all hypersurfaces $N$, see \cite[\S3.3.4]{ginzburg-duke-2007} and \cite[\S1.4]{cant-zhang-arXiv-2024}. Unfortunately, one has:
\begin{theorem}
  For all sufficiently small $\epsilon>0$, and $k<n$, the set:
  \begin{equation*}
    K=\bd B(a+\epsilon)\cap \set{\pi\abs{z_{i}}^{2}=2\epsilon\text{ for all $i=1,\dots,n-1$}}
  \end{equation*}
  links the ball and has capacity $2\epsilon$. Thus there are sets with arbitrarily small capacity which link the ball.
\end{theorem}
\begin{proof}
  The claim about the capacity of $K$ is well-known; indeed, the set is the elementary torus which has capacity $2\epsilon$ provided $\epsilon$ is small enough compared to $a$. The linking theorem follows from the failure of the packing inequality between $B(a)$ and $K$ in $B(a+\epsilon)$.
\end{proof}

\subsubsection{Widths of domains}
\label{sec:widths-domains}

Recently, \cite{buhovsky-tanny-JFPTA-2025} have established a general theorem which states that two domains $V_{1},V_{2}$ in $\R^{2n}$ satisfy a max \emph{inequality} provided that there exist sufficiently large disjoint thickenings $V_{i}(r)$ defined for $r\in [0,1]$. This means $V_{i}(r)$ is obtained by flowing $V_{i}$ by a vector field which remains outwardly transverse to $\bd V_{i}(r)$ along the flow. Their theorem is expressed in terms the \emph{width} of the thickening (see \cite[Definition 1.2]{buhovsky-tanny-JFPTA-2025}).

Unfortunately, their theorem is stated only for non-negative Hamiltonians (our argument relies on both non-negative and non-positive Hamiltonians). More seriously, the max \emph{inequality} is not sufficient for our argument to proceed (we require the max equality). This begs the question:
\begin{question}
  If $V_{1},V_{2}$ are disjoint subsets of $\C^{n}$ whose Buhovsky-Tanny widths are large enough, does it follow that $c(V_{1})+c(V_{2})\le \gamma(V_{1}\cup V_{2})$?
\end{question}
The thickenings considered in the above question should be disjoint.

A related concept to ``width'' is the maximal size of a disk cotangent bundle one can embed around a given Lagrangian $L$; see \cite{viterbo-CRASP-1990}.

\subsection{Acknowledgements}
\label{sec:acknowledgements}

The author would like to acknowledge H.\@~Alizadeh, M.\@ S.\@~Atallah, B.\@~Bramham, O.\@~Cornea, G.\@~Dimitroglou-Rizell, Y.\@~Ganor, V.\@~Humilière, S.\@~Matijevi\'c, L.\@~Nakamura, S.\@~Nemirovski, R.\@~Leclercq, S.\@~Seyfaddini, E.\@~Shelukhin, C.\@~Viterbo, and J.\@~Zhang for useful discussions. The author is supported in his research by funding from the ANR project CoSy.

\section{Proofs}
\label{sec:proofs}

We assume the reader is familiar with the set-up of spectral invariants and the spectral capacity, using the conventions of \cite{alizadeh-atallah-cant-arXiv-2024,cant-zhang-arXiv-2024}. To each compactly supported Hamiltonian isotopy $\varphi_{t}$, which is generated by a compactly supported family of smooth functions $H_{t}$, one associates the spectral invariant $c(\varphi_{t})$ of the unit element in Floer homology for $H_{t}$. For open sets $U$, one defines $c(U)=\sup \set{c(\varphi_{t}):\varphi_{t}\text{ is supported in }U}$, and for compact sets $K$, one defines $c(K)=\inf \set{c(U):K\subset U}$.

\subsection{Stability under Hamiltonian diffeomorphisms}
\label{sec:stab-under-hamilt}

One well-known fact that we will appeal to in our proof is:
\begin{lemma}
  Suppose that $g\in G$ and $\varphi_{t}$ is a compactly supported isotopy. Then there is an equality of spectral invariants:
  \begin{equation*}
    c(g\varphi_{t}g^{-1})=c(\varphi_{t}).
  \end{equation*}
\end{lemma}
\begin{proof}
This can be proved using, e.g., a continuity argument and the fact that the action spectrum of $g_{s}\varphi_{t}g_{s}^{-1}$ is independent of $s$, where $g_{s}$, $s\in [0,1],$ is a Hamiltonian isotopy generating $g$. See, e.g., \cite[Corollary 4.3]{viterbo-mathann-1992}
\end{proof}

\subsection{Max formula}
\label{sec:max-formula}

The version of the maximum formula we require is:
\begin{lemma}
  If $\varphi_{0,t},\varphi_{1,t}$ are compactly supported Hamiltonian isotopies, and their supports lie on opposite sides of a hyperplane, then:
  \begin{equation*}
    c(\varphi_{0,t}\varphi_{1,t})=\max\set{c(\varphi_{0,t}),c(\varphi_{1,t})}.
  \end{equation*}
\end{lemma}
\begin{proof}
  As explained in the introduction, this result is not new (and is proved in \cite{humiliere-le-roux-seyfaddini-2016,ganor-tanny-AGT-2023}), and is known as the \emph{maximum formula for spectral invariants}.
\end{proof}

\subsection{The spectral diameter of the ball}
\label{sec:spectr-diam-ball}

\begin{lemma}
  Suppose that $\varphi_{t}$ is a Hamiltonian isotopy supported in $B(a)$. Then:
  \begin{equation*}
    c(\varphi_{t})+c(\varphi_{t}^{-1})\le a.
  \end{equation*}
\end{lemma}
\begin{proof}
  This is proved in \cite{alizadeh-atallah-cant-arXiv-2024}.
\end{proof}

\subsection{Proof of Theorem \ref{theorem:main}}
\label{sec:proof-theor-refth}

We can now prove the main theorem. Suppose that $K_{1}, K_{2}\subset \Omega$ are disjoint compact sets, and $c(\Omega)=\gamma(\Omega)=a$. If they are Hamiltonian linked, then we are done. Therefore, assume they are Hamiltonian unlinked. Then there are disjoint open neighborhoods $U_{1},U_{2}$ of $K_{1},K_{2}$ which are also unlinked. Then:
\begin{equation*}
  c(K_{1})+c(K_{2})\le c(U_{1})+c(U_{2})=\sup c(\varphi_{1,t})+\sup c(\varphi_{2,t}),
\end{equation*}
where the supremums are over $\varphi_{i,t}$ supported in $U_{i}$. Now let $g\in G$ realize the unlinking of $U_{1},U_{2}$. Then \S\ref{sec:stab-under-hamilt} implies:
\begin{equation*}
  c(\varphi_{i,t})=c(g\varphi_{i,t}g^{-1})\text{ for }i=1,2.
\end{equation*}
On the other hand, since $g\varphi_{1,t}g^{-1}$ and $g\varphi_{2,t}^{-1}g^{-1}$ are supported on opposite sides of a hyperplane, \S\ref{sec:stab-under-hamilt} and \S\ref{sec:max-formula} imply:
\begin{equation*}  c(\varphi_{1,t}\varphi_{2,t}^{-1})=c(g\varphi_{1,t}\varphi_{2,t}^{-1}g^{-1})=\max\set{c(g\varphi_{1,t}g^{-1}),c(g\varphi_{2,t}^{-1}g^{-1})}.
\end{equation*}
Similarly:
\begin{equation*}  c(\varphi_{2,t}\varphi_{1,t}^{-1})=c(g\varphi_{2,t}\varphi_{1,t}^{-1}g^{-1})=\max\set{c(g\varphi_{2,t}g^{-1}),c(g\varphi_{1,t}^{-1}g^{-1})}.
\end{equation*}
Thus:
\begin{equation*}
  \begin{aligned}    c(\varphi_{1,t})+c(\varphi_{2,t})&=c(g\varphi_{1,t}g^{-1})+c(g\varphi_{2,t}g^{-1})\\ &\hspace{1cm}\le c(\varphi_{1,t}\varphi_{2,t}^{-1})+c(\varphi_{2,t}\varphi_{1,t}^{-1})\le a,
  \end{aligned}
\end{equation*}
where we have used $\gamma(\Omega)=a$ in the last line. Taking the supremum over $\varphi_{1,t},\varphi_{2,t}$ completes the proof.\hfill$\square$

\subsection{Proof of Lemma \ref{lemma:closing}}
\label{sec:proof-lemma-closing}

The closing lemma we stated concerns symplectically embedding Liouville domains $\Omega\subset \C^{n}$ which are assumed to be boundary minimal (e.g., $\Omega=B(a)$).

Let $\Omega(sf)\subset \Omega$ be the subdomain obtained by flowing backwards by $\Omega$'s intrinsic Liouville flow for time $sf$. Then $\lambda_{\Omega}|_{\bd\Omega(sf)}=e^{-sf}\alpha$.

It is well-known that the spectral capacity of $\Omega(sf)\subset \C^{n}$ is the period of some Reeb orbit $\gamma_{s}$ for $e^{-sf}\alpha$, plus $\Gamma(\gamma_{s})$ where $\Gamma=[\lambda_{\C^{n}}-\lambda_{\Omega}]$ is a cohomology class, i.e.,
\begin{equation*}
  \text{period of }\gamma_{s}+\Gamma(\gamma_{s})=c(\Omega(sf)).
\end{equation*}
The set of periods of orbits which \emph{do not} pass through $V$ is independent of $s$, and the set of periods is closed and nowhere dense. In addition, the values of $\Gamma$ are countable (when evaluated on elements of $H_{1}(\Omega,\Z)$). Thus, if $\gamma_{s}$ never passed through $V$, as $s$ ranges over $(0,1]$, then we would conclude that $c(\Omega(sf))$ was valued in a countable union of closed and nowhere dense sets. Since $s\mapsto c(\Omega(sf))$ is continuous (as is easy to prove by monotonicity and the definition of $c$ via outer regularity) it follows that $c(\Omega(sf))$ is independent of $s$, contradicting boundary minimality of $\Omega$.\hfill$\square$

\subsection{Proof of Corollary \ref{corollary:lagrangian-hbar}}
\label{sec:proof-corollary-lag-hbar}

For this we only need to appeal to \cite[Theorem 1.6]{hermann-BSMF-2004} or \cite[Theorem 2]{cant-zhang-arXiv-2024}, which relates the areas of holomorphic disks on Lagrangians to their spectral capacity. To show the statement is not vacuous, let us observe that one can take $K=\mathrm{pr}^{-1}(\mathbb{RP}^{n-1})$ where $\mathrm{pr}:\bd B(a)\to \mathbb{CP}^{n-1}$ is the Hopf fibration; see \cite{weinstein-book-1977,audin-CMH-1988,polterovich-math-Z-1991} \hfill$\square$

\subsection{Proof of Lemma \ref{lemma:starshaped-unlinked}}
\label{sec:proof-lemma-starshaped-unlinked}

This is a straightforward isotopy extension type result (see, e.g., \cite[\S3]{mcduff-salamon-book-2017}). Let us consider an embedding:
\begin{equation*}
  E=E_{1}\sqcup E_{2}:K_{1}\sqcup K_{2}\to \R^{2n}
\end{equation*}
where $K_{i}\subset \R^{2n}$ are compact starshaped sets (about the origin), and such that $E_{i}$ extends to a symplectic map on a neighborhood of $K_{i}\subset \R^{2n}$. Define:
\begin{equation*}
  E_{r}= e^{r}E(e^{-r}z_{1},e^{-r}z_{2}).
\end{equation*}
Then $E_{r}$ is a symplectic isotopy, and $\omega(\d E_{r}(-),\bd_{r}E_{r})$ are closed one-forms on a neighborhood of $K_{i}$. Thus there is a smooth family of functions $H_{r}$ such that:
\begin{equation*}
  E_{r}^{*}\d H_{r}=\omega(\d E_{r}(-),\bd_{r}E_{r})
\end{equation*}
holds on a neighborhood of $K_{1}\sqcup K_{2}$. It follows by the isotopy extension argument that the Hamiltonian flow by $X_{H_{r}}$ on the interval $[0,s]$ takes $E_{0}$ to $E_{s}$.

The image of $E_{s}$ is just a rescaled version of $E(e^{-s}K_{1},e^{-s}K_{2})$. The shrunken images $E_{1}(e^{-s}K_{1})$ and $E_{2}(e^{-s}K_{2})$ are eventually contained in disjoint balls around $E_{1}(0)$ and $E_{2}(0)$, and hence they can be separated by a hyperplane. Thus the rescaled images $e^{s}E_{1}(e^{-s}K_{1})$ and $e^{s}E_{2}(e^{-s}K_{2})$ are also separated by a hyperplane, as desired.\hfill$\square$

\subsection{Proof of Theorem \ref{theorem:toric}}
\label{sec:proof-theorem-toric}

Write $K=K_{k,b_{1},b_{2}}$ and let $K^{c}$ denote the complement of $K$ inside of $\bd B(a)$. It suffices to prove the case $b_{1}<b_{2}$.

Since $K\subset E(b_{2},\dots,b_{2},\infty,\dots)$ and $K\subset E(\infty,\dots,\infty,a-b_{1},\dots)$, it follows that:
\begin{equation*}
  c(K)\le \min\set{a-b_{1},b_{2}}.
\end{equation*}
Now observe that $K^{c}$ splits into two components:
\begin{enumerate}
\item $K^{c}_{1}=\set{\sum_{i=1}^{k} \pi \abs{z_{i}}^{2}\le b_{1}}\cap \bd B(a)$,
\item $K^{c}_{2}=\set{\sum_{i=k+1}^{n} \pi \abs{z_{i}}^{2}\le a-b_{2}}\cap \bd B(a)$.
\end{enumerate}
We study the Hamiltonian orbit of $K^{c}$. For simplicity, suppose $b_{1}\le a-b_{2}$. First translate:
\begin{equation*}
  K^{c}_{1}\mapsto K^{c}_{1}+Te_{n}\text{ and }K^{c}_{2}\mapsto K_{2}^{c},
\end{equation*}
for a large $T$. Then rotate by element of $g\in U(n)$ applied only to the translate of $K_{1}^{c}$, where $g$ takes $Te_{n}$ to $Te_{1}$ and $e_{1}$ to $-e_{n}$; this rotation isotopy occurs in the complement of $K_{2}^{c}$ provided $T$ was large enough. Thus we see that the Hamiltonian orbit of $K^{c}$ has a representative contained entirely in the region:
\begin{equation*}
  \set{\pi \abs{z_{n}}^{2}\le a-b_{2}},
\end{equation*}
and hence $c(K^{c})\le a-b_{2}$. The other case $a-b_{2}\le b_{1}$ is similar, and we conclude $c(K_{c})\le \max\set{a-b_{2},b_{1}}$.

Using the fact that $K$ and $K^{c}$ are Poisson commuting sets, we have:
\begin{equation*}
  a=c(K\cup K^{c})\le c(K)+c(K^{c})\le  \min\set{b_{2},a-b_{1}}+\max\set{a-b_{2},b_{1}}=a,
\end{equation*}
and so all inequalities must be equalities. This proves:
\begin{equation*}
  c(K)=\min\set{b_{2},a-b_{1}},
\end{equation*}
as desired.\hfill$\square$

\subsection{Proof of Theorem \ref{theorem:chord-1-2}}
\label{sec:proof-theorem-chord-1-2}

  The construction of \cite{mohnke-annals-2001} yields two disjoint Lagrangians inside $B(1)$ of capacities arbitrarily close to the minimal length of a Reeb chord of $\Lambda_{0}\cup \Lambda_{1}$. These Lagrangians are unlinked because the Legendrians are unlinked. Thus the packing inequality holds, and we conclude the minimal length of a Reeb chord is at most $1/2$.

  To see why the Lagrangians are unlinked, recall that they are defined as: $$L_{i,R}:=\Lambda_{i}\times \bd ([0,R]\times [\delta,1])\to B(1),$$ where $\lambda$ restricts to $x\d y$ where $x,y$ are coordinates on the rectangle. The isotopy:
  \begin{equation*}
    L_{i,s}:=e^{s/2}L_{i,e^{-s}R}
  \end{equation*}
  is an exact isotopy --- the area bounded by the loop remains constant. For $s$ large enough, the projections of the Lagrangians to the contact ideal boundary enter arbitrarily small neighborhoods of the Legendrians. By assumption the Legendrians are unlinked, and so one can unlink the Lagrangians for $s$ large enough via the symplectization lift of the isotopy unlinking $\Lambda_{i}$.\hfill$\square$

\bibliographystyle{./amsalpha-doi}
\bibliography{citations}
\end{document}